\documentclass{amsart}

\usepackage{amsmath}
\usepackage{mathtools}
\usepackage{amsthm}
\usepackage{amssymb}
\usepackage{hyperref}

\newcommand*\fullref[3][\relax]{%
  \ifdefined\hyperref%
    {\hyperref[#3]{#2\penalty 200\ \ref*{#3}#1}}%
  \else%
    {#2\penalty 200\ \relax\ref{#3}#1}%
  \fi%
}

\usepackage{tikz}

\usepackage{etoolbox}

\usetikzlibrary{calc}

\newcommand*\tikzremember[2]{\tikz[remember picture]\node[inner sep=0mm] (#1) {#2};}
\newcommand*\tikzmatrixline[3]{%
  \ifstrequal{#3}{l}{%
    \tikz[matrixoverlay]\draw[matrixline] ($ (#1.north west) + (-5pt,0) $) -- ($ (#2.south west) + (-5pt,-1pt) $);
  }{%
    \ifstrequal{#3}{r}{%
      \tikz[matrixoverlay]\draw[matrixline] ($ (#1.north east) + (5pt,0) $) -- ($ (#2.south east) + (5pt,-1pt) $);
    }{%
      \ifstrequal{#3}{t}{%
        \tikz[matrixoverlay]\draw[matrixline] ($ (#1.north west) + (-1pt,3.5pt) $) -- ($ (#2.north east) + (1pt,3.5pt) $);
      }{%
        \tikz[matrixoverlay]\draw[matrixline] ($ (#1.south west) + (-1pt,-3.5pt) $) -- ($ (#2.south east) + (1pt,-3.5pt) $);
      }%
    }%
  }%
}

\tikzset{
  matrixoverlay/.style={
    remember picture,
    overlay,
    matrixline/.style={
      thick,
      gray,
      yshift=-3.5pt,
    },
  },
}
\theoremstyle{definition}
\newtheorem{definition}{Definition}[section]

\theoremstyle{plain}
\newtheorem{corollary}[definition]{Corollary}
\newtheorem{lemma}[definition]{Lemma}
\newtheorem{proposition}[definition]{Proposition}
\newtheorem{theorem}[definition]{Theorem}

\numberwithin{equation}{section}
\newcommand*{\defterm}[1]{\emph{#1}}

\makeatletter
\newcommand\chyph{\penalty\@M-\hskip\z@skip}
\makeatother

\DeclarePairedDelimiter{\parens}{\lparen}{\rparen}

\DeclarePairedDelimiter{\set}{\{}{\}}
\DeclarePairedDelimiterX{\gset}[2]{\{}{\}}{\,#1:#2\,}

\makeatletter
\newcommand{\sizeddelimiter}[2]{\bBigg@{#1}#2}
\makeatother

\newcommand*{\nset}{\mathbb{N}}

\newcommand*{\rset}{\mathbb{R}}

\newcommand*{\tset}{\mathbb{T}}

\newcommand*{\emptyword}{\varepsilon}

\DeclarePairedDelimiterX{\pres}[2]{\langle}{\rangle}{#1\,\delimsize\vert\,\mathopen{}#2}

\newcommand*{\aA}{\mathcal{A}}

\newcommand*{\plac}{{\mathsf{plac}}}

\newcommand*{\placcong}{\equiv_\plac}

\newcommand*{\plit}{\mathrm{P}}

\newcommand*{\pplac}[2][]{\plit_{\plac}\parens[#1]{#2}}

\begin{document}

\title[A note on identities]{A note on identities in plactic monoids and monoids of upper-triangular tropical matrices}

\author[A.J. Cain]{Alan J. Cain}
\address{%
Centro de Matem\'{a}tica e Aplica\c{c}\~{o}es (CMA)\\
Faculdade de Ci\^{e}ncias e Tecnologia\\
Universidade Nova de Lisboa\\
2829--516 Caparica\\
Portugal
}
\email{%
a.cain@fct.unl.pt
}

\thanks{The first author was supported by an Investigador {\sc FCT} fellowship ({\sc IF}/01622/2013/{\sc CP}1161/{\sc
    CT}0001).}

\author[G. Klein]{Georg Klein}
\address{%
  Institute of Mathematics\\
  University of Warsaw\\
  Banacha~2\\
  02--097 Warsaw\\
  Poland
}
\email{%
  gklein@mimuw.edu.pl
}

\author[\L. Kubat]{\L ukasz Kubat}
\address{%
  Institute of Mathematics\\
  University of Warsaw\\
  Banacha~2\\
  02--097 Warsaw\\
  Poland
}
\email{%
  lukasz.kubat@mimuw.edu.pl
}

\author[A. Malheiro]{Ant\'{o}nio Malheiro}
\address{%
Centro de Matem\'{a}tica e Aplica\c{c}\~{o}es (CMA)\\
Faculdade de Ci\^{e}ncias e Tecnologia\\
Universidade Nova de Lisboa\\
2829--516 Caparica\\
Portugal
}
\address{%
Departamento de Matem\'{a}tica\\
Faculdade de Ci\^{e}ncias e Tecnologia\\
Universidade Nova de Lisboa\\
2829--516 Caparica\\
Portugal
}
\email{%
ajm@fct.unl.pt
}

\thanks{For the first and fourth authors, this work was partially supported by by the Funda\c{c}\~{a}o para a
  Ci\^{e}ncia e a Tecnologia (Portuguese Foundation for Science and Technology) through the project {\sc UID}/{\sc
    MAT}/00297/2013 (Centro de Matem\'{a}tica e Aplica\c{c}\~{o}es) and the project {\scshape PTDC}/{\scshape
    MHC-FIL}/2583/2014. Much of the research leading to this paper was undertaken during visits by the first and fourth
  authors to the University of Warsaw and these authors thank the university for its hospitality.}

\author[J. Okni\'{n}ski]{Jan Okni\'{n}ski}
\address{%
  Institute of Mathematics\\
  University of Warsaw\\
  Banacha~2\\
  02--097 Warsaw\\
  Poland
}
\email{%
  okninski@mimuw.edu.pl
}

\begin{abstract}
  This paper uses the combinatorics of Young tableaux to prove the plactic monoid of infinite rank does not satisfy a
  non-trivial identity, by showing that the plactic monoid of rank $n$ cannot satisfy a non-trivial identity of length
  less than or equal to $n$. A new identity is then proven to hold for the monoid of $n \times n$ upper-triangular
  tropical matrices. Finally, a straightforward embedding is exhibited of the plactic monoid of rank~$3$ into the direct
  product of two copies of the monoid of $3\times 3$ upper-triangular tropical matrices, giving a new proof that the
  plactic monoid of rank~$3$ satisfies a non-trivial identity.
\end{abstract}

\maketitle

\section{Introduction}

Whether the plactic monoid (the monoid of Young tableaux, which is famed for its ubiquity
\cite[Chapter~5]{lothaire_algebraic}) and its finite-rank variants satisfy non-trivial identities is an
actively-researched open question \cite{kubat_identities,izhakian_tropical}. Various monoids that are related to the
plactic monoid satisfy non-trivial identities. For instance, the Chinese monoid \cite{cassaigne_chinese}, which has the
same growth type as the plactic monoid \cite{duchamp_placticgrowth}, satisfies Adian's identity
$xyyxxyxyyx = xyyxyxxyyx$ \cite[Corollary~3.3.4]{jaszunska_chinese}. (This is the shortest non-trivial identity
satisfied by the bicyclic monoid \cite[Chapter~IV, Theorem~2]{adian_defining}.) Furthermore, various `plactic-like'
monoids of combinatorial objects satisfy non-trivial identities \cite{cm_identities}: for instance, the hypoplactic
monoid (the monoid of quasi-ribbon tableaux \cite{krob_noncommutative4,novelli_hypoplactic}), the sylvester monoid
(binary search trees \cite{hivert_sylvester}), and the Baxter monoid (pairs of twin binary search trees
\cite{giraudo_baxter2}).

Let $\plac$ denote the (infinite-rank) plactic monoid and let $\plac_n$ denote the plactic monoid of rank $n$, for any
$n \in \nset$. First, $\plac_1$ is monogenic and thus commutative and so trivially satisfies $xy = yx$. The monoid
$\plac_2$ is isomorphic to the Chinese monoid of rank $2$ and so satisfies Adian's identity. The third and fifth authors
proved that $\plac_3$ satisfies the identity $uvvuvu = uvuvvu$, where $u(x,y)$ and $v(x,y)$ are respectively the left
and right side of Adian's identity (and so the identity $uvvuvu = uvuvvu$ has sixty variables $x$ or $y$ on each side)
\cite[Theorem~2.6]{kubat_identities}. Furthermore, $\plac_3$ does not satisfy Adian's identity
\cite[p.~111--2]{kubat_identities}. Izhakian developed a theory that allows one to show that $\plac_3$ is isomorphic to
a subdirect product of two copies of the monoid of $3 \times 3$ upper-triangular tropical matrices $U_3(\tset)$
\cite{izhakian_tropical}. Since the same author had already proven that for any $n \in \nset$ the monoid of $n \times n$
upper-triangular tropical matrices $U_n(\tset)$ satisfies a non-trivial identity \cite{izhakian_identitiestriangular},
this gives an alternative proof that $\plac_3$ satisfies a non-trivial identity. It remains open whether $\plac_n$
satisfies a non-trivial identity for $n \geq 4$.

This note contributes further to this topic. First, \fullref{Section}{sec:plac} uses the combinatorics of Young tableaux
to prove that the infinite-rank plactic monoid $\plac$ does not satisfy a non-trivial identity, in contrast to the
Chinese, hypoplactic, sylvester, and other monoids discussed above. Since $\plac$ is the union of all the
finite-rank plactic monoids $\plac_n$, this implies that there is no non-trivial identity satisfied by all the
$\plac_n$. However, it is still possible that there is a hierarchy of identities satisfied by the various
$\plac_n$. \fullref{Section}{sec:uppertriangular} gives a new, simple, identity satisfied by the monoid of $n \times n$
upper-triangular tropical matrices $U_n(\tset)$; the proof is based on an approach developed by the fifth author
\cite{okninski_identities}. Finally, \fullref{Section}{sec:troprep} gives an elementary proof that $\plac_3$ embeds into
the direct product of two copies of $U_3(\tset)$; this gives another proof that $\plac_3$ satisfies a non-trivial
identity.

\section{Preliminaries and notation}
\label{sec:preliminaries}

This section briefly recalls some definitions and sets up notation; see the references for further background.

An \defterm{identity} is a formal equality between two words in the free monoid, and is \defterm{non-trivial} if the two
words are distinct. A monoid $M$ \defterm{satisfies} such an identity if the equality holds in $M$ under every
substitution of letters in the words by elements of $M$. For example, any commutative monoid satisfies the non-trivial
identity $xy = yx$.

For background on the plactic monoid, Young tableaux, and Schensted's algorithm, see
\cite[Ch.5]{lothaire_algebraic}. Let $\aA = \set{1 < 2 < \dotsb}$ be the set of natural numbers viewed as an infinite ordered alphabet. For
$n \in \nset$, let $\aA_n = \set{1 < 2 < \dotsb < n}$ be the set of the first $n$ natural numbers viewed as a finite
ordered alphabet. For any $u \in \aA^*$, let $\pplac{u}$ be the Young tableaux (the `plactic $\plit$-symbol') computed
from $u$ using Schensted's algorithm. The plactic congruence $\placcong$ relates words in $\aA^*$ that have the same
image under $u \mapsto \pplac{u}$. The \defterm{plactic monoid} $\plac$ is the factor monoid $\aA^*\!/{\placcong}$; the
\defterm{plactic monoid of rank $n$} is the factor monoid $\aA_n^*/{\placcong}$ (where $\placcong$ is naturally restricted
to $\aA_n^*$).

The tropical semiring is $\tset = \rset \cup \set{-\infty}$, and its addition and multiplication operations $\oplus$ and
$\otimes$ are defined by $x \oplus y = \max\set{x,y}$ and $x \otimes y = x + y$. The set of $n \times n$
upper-triangular tropical matrices is denoted $U_n(\tset)$. (An upper-triangular tropical matrix has all entries below
the main diagonal equal to $-\infty$, since this is the additive identity of $\tset$.) For $X,Y \in U_n(\tset)$, write
$X \sim_{\mathrm{diag}} Y$ if $X$ and $Y$ have equal corresponding diagonal entries.

Let $\psi_1\colon U_n(\tset) \to U_{n-1}(\tset)$ be the map that sends a matrix $Z \in U_n(\tset)$ to the matrix in
$U_{n-1}(\tset)$ obtained by erasing the rightmost column and bottommost row of $Z$.  Let
$\psi_2\colon U_n(\tset) \to U_{n-1}(\tset)$ be the map that sends a matrix $Z \in U_n(\tset)$ to the matrix in
$U_{n-1}(\tset)$ obtained by erasing the leftmost column and topmost row of $Z$. The maps $\psi_1$ and $\psi_2$ are
monoid homomorphisms.

\section{The infinite-rank plactic monoid}
\label{sec:plac}

\begin{proposition}
  \label{prop:plackidminlength}
  The plactic monoid $\plac_n$ does not satisfy any non-trivial identity of length less than or equal to $n$.
\end{proposition}

\begin{proof}
  Suppose, with the aim of obtaining a contradiction, that $\plac_n$ satisfies a non-trivial identity of length less
  than or equal to $n$. Without loss of generality, assume that it satisfies such an identity over the variable set
  $\set{x,y}$, say $u(x,y) = v(x,y)$ of length equal to $n$ (that is, with $|u| = |v| = n$). Interchanging $u$ and $v$
  if necessary, assume $u$ is lexicographically less than $v$ (where the variable set is ordered by $x < y$). Suppose
  $u(x,y) = u_1\cdots u_n$ and $v(x,y) = v_1\cdots v_n$, where each $u_i$ and $v_i$ is a variable $x$ or $y$. If $j$ is
  minimal such that $u_j$ and $v_j$ are different variables, then $u_j$ is $x$ and $v_j$ is $y$.

  Let $s = 12\cdots n \in \aA_n^*$ and let $t = 12\cdots (n-j)(n-j+2)\cdots n \in \aA_n^*$. (So the word $t$ does not
  contain the generator $n-j+1$.)  Since $\plac_n$ satisfies the identity $u(x,y) = v(x,y)$, the equality
  $u\parens[\big]{\pplac{s},\pplac{t}} = v\parens[\big]{\pplac{s},\pplac{t}}$ holds. Thus the Young tableaux
  $\pplac[\big]{u(s,t)}$ and $\pplac[\big]{v(s,t)}$ are equal. In particular, the height of the leftmost columns of
  these Young tableaux are equal. By Schensted's theorem \cite{schensted_longest}, the height of the leftmost column of
  $\pplac{w}$ (for $w \in \aA^*$) is equal to the length of the longest strictly decreasing subsequence in $w$. Thus
  the lengths of the longest strictly decreasing subsequences of $u(s,t)$ and $v(s,t)$ are equal.

  Since $s$ and $t$ only contain symbols in $\set{1,\ldots,n}$, the longest strictly decreasing subsequences of $u(s,t)$
  and $v(s,t)$ cannot be longer than $n$.

  In $u(s,t)$, there is a strictly decreasing subsequence of length $n$, namely $n(n-1)\cdots 21$, where, for
  $i \in \set{1,\ldots,n}$, the symbol $i$ is chosen from the instance of $s$ or $t$ that was substituted for
  $u_{n-i+1}$. Note in particular that since $n-(n-j+1)+1 = j$, the symbol $n-j+1$ is chosen from the instance of $s$
  substituted for $u_j$ (which is the variable $x$). By the previous paragraph, $n(n-1)\cdots 21$ is a \emph{longest}
  strictly decreasing subsequence of $u(s,t)$.

  Thus $v(s,t)$ must also have a strictly decreasing subsequence of length $n$. Since each word $s$ or $t$ is strictly
  increasing, such a subsequence contains at most one symbol chosen from the instance of $s$ or $t$ substituted for each
  $v_i$. Since $|v| = n$, this implies that \emph{exactly} one symbol is chosen from each instance of $s$ or $t$. Hence,
  for $i \in \set{1,\ldots,n}$, the symbol $i$ must be chosen from the instance of $s$ or $t$ that was substituted for
  $v_{n-i+1}$. In particular, since $n-(n-j+1)+1 = j$, the symbol $n-j+1$ is chosen from the instance of $t$
  substituted for $v_j$ (which is the variable $y$). This is a contradiction, because the word $t$ does not contain a
  symbol $n-j+1$.

  Therefore $\plac_n$ does not satisfy a non-trivial identity of length less than or equal to $n$.
\end{proof}

Since every finite-rank plactic monoid is a submonoid of the infinite-rank plactic monoid, the following result is
immediate from \fullref{Proposition}{prop:plackidminlength}:

\begin{theorem}
  The infinite-rank plactic monoid $\plac$ does not satisfy any non-trivial identity.
\end{theorem}

\section{A new identity for the monoid of upper-triangular tropical matrices}
\label{sec:uppertriangular}

Define
\begin{align*}
u_0(p,q) &= p, & v_0(p,q) &= q,\\
u_1(p,q) &= pqppq, & v_1(p,q) &= pqqpq,
\end{align*}
and inductively define
\begin{equation*}
u_n(p,q) = u_1\parens[\big]{u_{n-1}(p,q),v_{n-1}(p,q)}, \quad v_n(p,q) =  v_1\parens[\big]{u_{n-1}(p,q),v_{n-1}(p,q)}.
\end{equation*}
for $n \geq 2$. Note that the preceding equations hold for $n = 1$, but we only use them as a definition for $n \geq
2$. Note also that $u_1(xy,yx) = v_1(xy,yx)$ is Adian's identity.

\begin{proposition}
  \label{prop:undiagid}
  Let $X,Y \in U_n(\tset)$ with $X \sim_{\mathrm{diag}} Y$. Then $u_{n-1}(X,Y) = v_{n-1}(X,Y)$.
\end{proposition}

The proof strategy is similar to that used by the fifth author to show that $U_n(\tset)$ satisfied a different set of
identities \cite{okninski_identities}.

\begin{proof}
  Proceed by induction on $n$. For $n = 1$, since the only entry of a $1 \times 1$ matrix is on its main diagonal, it
  follows from $X \sim_{\mathrm{diag}} Y$ that $X = Y$; hence $u_0(X,Y) = X = Y = v_0(X,Y)$. This is the base of the induction.

  For the induction step, let $n \geq 2$ and assume the result holds for $n-1$. Let $A = u_{n-2}(X,Y)$ and
  $B = v_{n-2}(X,Y)$.  Recall the homomorphisms $\psi_1,\psi_2\colon U_n(\tset) \to U_{n-1}(\tset)$ defined in
  \fullref{Section}{sec:preliminaries}. By the induction hypothesis,
  \begin{align*}
    u_{n-2}\parens[\big]{\psi_1(X),\psi_1(Y)} &= v_{n-2}\parens[\big]{\psi_1(X),\psi_1(Y)}\text{ and } \\
    u_{n-2}\parens[\big]{\psi_2(X),\psi_2(Y)} &= v_{n-2}\parens[\big]{\psi_2(X),\psi_2(Y)}. \\
    \intertext{Thus}
    \psi_1(A) = \psi_1\parens[\big]{u_{n-2}(X,Y)} &= \psi_1\parens[\big]{v_{n-2}(X,Y)} = \psi_1(B)\text{ and } \\
    \psi_2(A) = \psi_2\parens[\big]{u_{n-2}(X,Y)} &= \psi_2\parens[\big]{v_{n-2}(X,Y)} = \psi_2(B).
  \end{align*}
  Therefore the matrices $A$ and $B$ can differ only in their $(1,n)$-th entries.

  Denote the matrix $C = \begin{bmatrix}c_{ij}\end{bmatrix}$ by $(c_{11},c_{1n},c_{nn};\kappa)$ for some suitable array
  $\kappa$ representing the rest of $C$. That is,
  \[
    (c_{11},c_{1n},c_{nn};\kappa)\quad\text{denotes}\quad
    \begin{bmatrix}
      \tikzremember{cnotattl}{$\phantom{c_{nn}}\mathllap{c_{11}}$} &        & \tikzremember{cnotattr}{$\mathrlap{c_{1n}}\phantom{c_{nn}}$} \\[2mm]
                                        & \kappa &                                   \\[2mm]
                                        &        & \tikzremember{cnotatbr}{$c_{nn}$} \\
    \end{bmatrix}.
    \tikz[matrixoverlay]\draw[matrixline] ($ (cnotattl.south west) + (-1pt,-3.5pt) $) -- ($ (cnotattl.south east) + (1pt,-3.5pt) $)  -- ($ (cnotattl.north east) + (1pt,2.5pt) $);
    \tikz[matrixoverlay]\draw[matrixline] ($ (cnotattr.south east) + (1pt,-3.5pt) $) -- ($ (cnotattr.south west) + (-1pt,-3.5pt) $)  -- ($ (cnotattr.north west) + (-1pt,2.5pt) $);
    \tikz[matrixoverlay]\draw[matrixline] ($ (cnotatbr.north east) + (1pt,3.5pt) $) -- ($ (cnotatbr.north west) + (-1pt,3.5pt) $)  -- ($ (cnotatbr.south west) + (-1pt,-2.5pt) $);
  \]
  Note that if $D = (d_{11},d_{1n},d_{nn};\mu)$, then
  \[
    CD = \parens[\big]{c_{11} + d_{11},\max\set{c_{11}+d_{1n},c_{1n}+d_{nn},\gamma},c_{nn}+d_{nn};\nu},
  \]
  for some $\gamma$ and some array $\nu$. In fact, $\gamma$ is the maximum of the other sums of pairs in the scalar
  product of the first row of $C$ and first column of $D$ and so is only dependent on $\kappa$ and $\mu$ (and not on
  $c_{11}$, $c_{1n}$, $c_{nn}$, $d_{11}$, $d_{1n}$, or $d_{nn}$) and $\nu$ is only dependent on $\kappa$, $\mu$,
  $c_{11}$ and $d_{nn}$ (and not on $c_{1n}$, $c_{nn}$, $d_{11}$, or $d_{1n}$).

  So $A = (c,a,d;\kappa_1)$ and $B = (c,b,d;\kappa_1)$ for some $c,a,b,d \in \tset$ and some array $\kappa_1$. Then
  \[
    AB = \parens[\big]{2c,\max\set{b+c,a+d,\gamma_2},2d,\kappa_2}
  \]
  for some $\gamma_2 \in \tset$ and array $\kappa_2$. Next,
  \begin{align*}
    ABA &= (AB)A = \parens[\big]{3c,\max\set{a+2c,b+c+d,a+2d,d+\gamma_2,\gamma_3},3d;\kappa_3}, \\
    ABB &= (AB)B = \parens[\big]{3c,\max\set{b+2c,b+c+d,a+2d,d+\gamma_2,\gamma_3},3d;\kappa_3};
  \end{align*}
  for some $\gamma_3 \in \tset$ and array $\kappa_3$; note that the these are equal in $(AB)A$ and $(AB)B$ because they depend
  only on equal corresponding entries of $A$ and $B$. Finally,
  \begin{align*}
    ABAAB = (ABA)AB &= \parens[\big]{5c,\max\set{b+4c,a+3c+d,3c+\gamma_2,a+2c+2d,\\
    &\qquad\qquad b+c+3d,a+4d,4d+\gamma_2,2d+\gamma_3,\gamma_5},5d;\kappa_5}, \\
    ABBAB = (ABB)AB &= \parens[\big]{5c,\max\set{b+4c,a+3c+d,3c+\gamma_2,b+2c+2d,\\
    &\qquad\qquad b+c+3d,a+4d,4d+\gamma_2,2d+\gamma_3,\gamma_5},5d;\kappa_5};
  \end{align*}
  for some $\gamma_5 \in \tset$ and array $\kappa_5$; note that these are equal in $(ABA)AB$ and $(ABB)AB$ because they depend
  only on equal corresponding entries of $ABA$ and $ABB$. Note that the expressions for $ABAAB$ and $ABBAB$ only differ
  in the terms $a+2c+2d$ and $b+2c+2d$, respectively.

  Now consider the four possible relative orders of $a$ and $b$ and of $c$ and $d$. In each case, there is some element
  that appears in the $\max\set{\ldots}$ terms of the expressions for both $ABAAB$ and $ABBAB$ and that is greater than
  both $a+2c+2d$ and $b+2c+2d$:
  \begin{align*}
    a \geq b \land c \geq d &\implies a+3c+d \geq a+2c+2d,\;\;b+2c+2d; \\
    a \geq b \land c \leq d &\implies a+4d \geq a+2c+2d,\;\;b+2c+2d; \\
    a \leq b \land c \geq d &\implies b+4c \geq a+2c+2d,\;\;b+2c+2d; \\
    a \leq b \land c \leq d &\implies b+c+3d \geq a+2c+2d,\;\;b+2c+2d.
  \end{align*}
  So in each case, $a+2c+2d$ and $b+2c+2d$ are not the maximum terms in the expressions for $ABAAB$ and $ABBAB$. Thus
  $u_1(A,B) = ABAAB = ABBAB = v_1(A,B)$. Therefore $u_{n-1}(X,Y) = u_1(A,B) = v_1(A,B) = v_{n-1}(X,Y)$.

  Hence, by induction, the result holds for all $n$.
\end{proof}

\begin{theorem}
  \label{thm:tropicalid}
  The monoid of $n \times n$ upper-triangular tropical matrices $U_n(\tset)$ satisfies the identity
  $u_{n-1}(xy,yx) = v_{n-1}(xy,yx)$.
\end{theorem}

\begin{proof}
  Let $X,Y \in U_n(\tset)$. Then $XY \sim_{\mathrm{diag}} YX$; the result is immediate from
  \fullref{Proposition}{prop:undiagid}.
\end{proof}

\section{A tropical representation of $\plac_3$}
\label{sec:troprep}

For any $w \in \aA^+$ and any $p,q\in \aA$, define:
\[
  w_{pq}=
  \begin{cases}
    \parbox{8cm}{\raggedright the maximal length of a nondecreasing subsequence in $w$ with entries in the interval $[p,q]$} &
    \text{if $p\leq q$,}\\
    -\infty & \text{if $p>q$.}
  \end{cases}
\]
For the identity $\emptyword$ of $\aA^*$, define $\emptyword_{pq}=0$ if $p=q$ and $\emptyword_{pq} = -\infty$ otherwise.

Define a map $\phi_n\colon \aA^* \to U_n(\tset)$ by $\phi_n(w) = \begin{bmatrix}w_{pq}\end{bmatrix}$.

\begin{lemma}
  The map $\phi_n$ is a homomorphism.
\end{lemma}

\begin{proof}
  Proceed by induction on $n$. If $n=1$ then the claim is clear. Assume that $n \geq 2$ and recall the homomorphisms
  $\psi_1,\psi_2\colon U_n(\tset) \to U_{n-1}(\tset)$ defined in \fullref{Section}{sec:preliminaries}. Let
  $\vartheta_1 = \psi_1\circ\phi_n$ and $\vartheta_2 = \psi_2\circ\phi_n$. Clearly, $\vartheta_1$ is simply $\phi_{n-1}$ and
  is thus a homomorphism by the induction hypothesis. Furthermore, if $w \mapsto \overline{w}$ is the map from
  $\set{2,\ldots,n}^*$ to $\aA_{n-1}^*$, extending $a \mapsto \overline{a-1}$, then
  $\vartheta_2(w) = \phi_{n-1}(\overline{w})$ and so is a homomorphism by the induction hypothesis.

  By the definition of matrix multiplication in $U_{n}(\tset)$, it now remains to prove that for every $u,v\in P_n$ the
  $(1,n)$-th entries of $\phi_n(uv)$ and $\phi_n(u)\phi_n(v)$ are equal for all $u,v \in \aA_n^*$. In other words,
  \begin{equation}
    \label{eq:homom1}
    (uv)_{1n} =\max\set{u_{11}+v_{1n},u_{12}+v_{2n},\ldots, u_{1n}+v_{nn}}.
  \end{equation}
  However, it is clear that the longest nondecreasing subsequence in the product $uv$ is of the form $\alpha\beta$,
  where for some $j$ the word $\alpha$ is the longest nondecreasing subsequence in $u$ with entries in the interval
  $[1,j]$ and $\beta$ is the longest nondecreasing subsequence of $v$ with entries in the interval $[j,n]$. Therefore
  \eqref{eq:homom1} indeed holds and the result follows.

  Thus, by induction on $n$, the map $\phi_n$ is a homomorphism from $\aA_n^*$ to $U_n(\tset)$.
\end{proof}

\begin{lemma}
  The map $\phi_n$ factors to give a homomorphism $\phi_n\colon \plac_n \to U_n(\tset)$.
\end{lemma}

\begin{proof}
  Notice that the plactic relations $(baa,aba)$, $(bab,bba)$, $(bca,bac)$, $(cab,acb)$ (for $a < b < c$), with
  $a,b,c \in \aA$, preserve $w_{pq}$ for any $w \in \aA^{*}$ and any $p,q \in \aA$. This is easy to see by a direct
  verification of the possible cases: the intersection $\set{a,b,c} \cap [p,q]$ is either a single generator, or is
  equal to $\set{a,b}$, $\set{b,c}$ or $\set{a,b,c}$. It follows that if $u,v \in \aA_n^*$ are such that
  $u \placcong v$, then $\phi_n(u) = \phi_n(v)$.
\end{proof}

Define $f_n\colon \aA_n^* \to \aA_n^*$ as follows. For any $i \in \aA_n$, define
\[
  f_n(i) = n (n-1) \dotsm (i+1)(i-1) \dotsm 1,
\]
and for $w = a_{1}\dotsm a_{k}\in \aA_n^*$, where $a_{i} \in \aA$, define $f_n(w) = f_n(a_{k})\dotsm f_n(a_{1})$. It is
straightforward to prove that the antihomomorphism $f_n$ factors to give an antihomomorphism $f_n\colon \plac_n \to \plac_n$.


Let $F = \begin{bmatrix}F_{ij}\end{bmatrix}$ be the $n \times n$ tropical matrix with $F_{i,n+1-i} = 0$ for
$i=1,2,\ldots, n$ and all other entries being $-\infty$. Define an involution $\pi_n\colon U_{n}(\tset) \to U_{n}(\tset)$ by
$x \mapsto (FxF)^{T}$ (conjugation by $F$ composed with transposition); note that $\pi_n$ is an antihomomorphism (and in
fact an antiautomorphism, since it is bijective). Define $\sigma_n = \pi_n \circ \phi_n \circ f_n$; since $f_n$ and
$\pi_n$ are antihomomorphisms, $\sigma_n$ is a homomorphism.

Finally, define the homomorphism $\Psi_3\colon \plac_3 \to U_3(\tset) \times U_3(\tset)$ by
$\Psi_3(u) = \parens[\big]{\phi_3(u),\sigma_3(u)}$.

\begin{proposition}
  The map $\Psi_3$ is an embedding.
\end{proposition}

\begin{proof}
An arbitrary element of $\plac_3$ has a unique representation by a word of the form
$3^\zeta 2^\delta 3^\epsilon 1^\alpha 2^\beta 3^\gamma$, where $\zeta \leq \delta \leq \alpha$ and
$\delta+\epsilon \leq \alpha+\beta$. (This is the `row reading' of a Young tableau.) Notice that, by the definition of
$\phi_3$,
\begin{equation}
  \label{eq:phirowreading}
  \phi_3(3^\zeta 2^\delta 3^\epsilon 1^\alpha 2^\beta 3^\gamma) =
  \begin{bmatrix}
    \alpha  & \alpha + \beta & \alpha+\beta+\gamma                        \\
    -\infty & \beta+\delta   & \max\set{\beta+\epsilon} + \delta + \gamma \\
    -\infty & -\infty        & \zeta + \epsilon + \delta                  \\
  \end{bmatrix}.
\end{equation}

An arbitrary element of $\plac_3$ also has a unique representative of the form
\[
  u=(321)^{k_1}(21)^{k_2}(31)^{k_3}1^{k_4}(32)^{k_5}2^{k_6}3^{k_7},
\]
for some $k_i \in \nset \cup \set{0}$ and either $k_4 = 0$ or $k_5 = 0$. (This is the column reading of a Young
tableau.) Straightforward computation shows that
\[
  f_3(u)=\begin{cases}
    (21)^{k_7}(31)^{k_6} (31)^{k_5}(21)^{k_5}\\
    \qquad(32)^{k_3}(21)^{k_3}(32)^{k_2}(31)^{k_2}(32)^{k_1}(31)^{k_1}(21)^{k_1} &\text{if $k_4=0$},\\
    (21)^{k_7}(31)^{k_6} (32)^{k_4}(32)^{k_3}\\
    \qquad(21)^{k_3}(32)^{k_2}(31)^{k_2}(32)^{k_1}(31)^{k_1}(21)^{k_1} & \text{if $k_5=0$};
  \end{cases}
\]
thus
\[
  f_3(u)=\begin{cases}
    (321)^{2k_1+k_2+k_3+k_5}(21)^{k_7}(31)^{k_6}1^{k_5}2^{k_3}3^{k_2} & \text{if $k_4=0$},\\
    (321)^{2k_1+k_2+k_3}(21)^{k_7}(31)^{k_6} (32)^{k_4}2^{k_3}3^{k_2} & \text{if $k_5=0$}.
  \end{cases}
\]
Regardless of whether $k_4 = 0$ or $k_5 = 0$, the row reading of $u$ (or more precisely $\pplac{u}$) has the form:
\[
  3^{k_1}2^{k_1+k_2}3^{k_3+k_5}1^{k_1+k_2+k_3+k_4}2^{k_5+k_6}3^{k_7}.
\]
By \eqref{eq:phirowreading}, $\phi_3(u)$ determines the following numbers:
\begin{gather}
  \label{eq:phidetermined}
  \begin{gathered}
    k_7,\quad k_5+k_6,\quad k_1+k_2+k_3+k_4,\\
    k_1+k_2,\quad k_1+k_3+k_5,\quad\max \set{k_3+k_5, k_5+k_6}.
  \end{gathered}
\end{gather}
The row reading of $f(u)$ (or more precisely $\pplac[\big]{f(u)}$) has the form:
\[
  3^{2k_1+k_2+k_3+k_5}2^{2k_1+k_2+k_3+k_5+k_7}3^{2k_1+k_2+k_3+k_4+k_5+k_6}
  1^{2k_1+k_2+k_3+2k_5+k_6+k_7}2^{k_3+k_4}3^{k_2}.
\]
Therefore, since $\pi_n$ is bijective, we know by \eqref{eq:phirowreading} that $\sigma_3(u)$ determines the following
numbers:
\begin{gather}
  \label{eq:sigmadetermined}
  \begin{gathered}
    k_{2},\quad k_{3}+k_4,\quad 2k_1+k_2+k_3+2k_5+k_6+k_7,\\
    2k_1+k_2+k_3+k_5+k_7,\quad
    4k_1+2k_2+2k_3+k_4+2k_5+k_6,\\
    \max\set{k_3+k_4,2k_1+k_2+k_3+k_4+k_5+k_6}.
  \end{gathered}
\end{gather}
Consider another elemenent $u' \in \plac_3$, with
\[
  u'=(321)^{k'_1}(21)^{k'_2}(31)^{k'_3}1^{k'_4}(32)^{k'_5}2^{k'_6}3^{k'_7},
\]
and suppose that $\phi_3(u) = \phi_3(u')$ and $\sigma_3(u) = \sigma_3(u')$. By \eqref{eq:phidetermined}, $k_7 =
k'_7$. By \eqref{eq:sigmadetermined}, $k_2 = k'_2$. Consequently, \eqref{eq:phidetermined} implies that $k_1 = k'_1$. Thus,
\[
  k_3+k_4= k_3'+k_4',\quad k_3+ k_5 = k_3'+ k_5',\quad k_5+k_6= k_5'+k_6',
\]
respectively by \eqref{eq:phidetermined}, \eqref{eq:sigmadetermined}, and \eqref{eq:phidetermined} together with $k_1 = k'_1$.

If $k_4=k_4'=0$, then $k_3 = k_3'$, and so $k_5 = k_5'$, and so $k_6 =k_6'$; thus $u=u'$. Similarly, if $k_5=k_5'=0$,
then $k_6 = k_6'$ and $k_3 = k_3'$ , and so $k_6 =k_6'$; thus $u=u'$.

So suppose that $k_4 = 0$ and $k_5' = 0$. Then
\[
  k_3 = k_3 + k_4 = k_3' + k_4' = k_3' + k_5' + k_4' = k_3 + k_5 + k_4';
\]
so $k_5 + k_4' = 0$ and so $k_5 = 0$. Hence $k_6 = k_6'$, and $k_3 = k_3'$ and so $k_4 = k_4'$. Therefore $u =
u'$. Parallel reasoning shows that $k'_4 = 0$ and $k_5 = 0$ implies $u = u'$.

Therefore $\phi_3(u) = \phi_3(u')$ and $\sigma_3(u) = \sigma_3(u')$ implies $u = u'$.

Hence the map $\Psi_3\colon \plac_3 \to U_3(\tset) \times U_3(\tset)$ with
$\Psi_3(u) = \parens[\big]{\phi_3(u),\sigma_3(u)}$ is an embedding.
\end{proof}

Since $\Psi_3$ embeds $\plac_3$ into the direct product of two copies of $U_3(\tset)$, the following result is immediate
from \fullref{Theorem}{thm:tropicalid}:

\begin{corollary}
$\plac_3$ satisfies $u_2(xy,yx) = v_2(xy,yx)$.
\end{corollary}

\bibliography{\jobname}
\bibliographystyle{alphaabbrv}

\end{document}